\documentclass[11pt,reqno]{amsart}
\numberwithin{equation}{section}
\usepackage{pb-diagram, graphics}
\usepackage{setspace}
\usepackage{amscd,amsfonts}
\usepackage{amssymb}
\usepackage{amsmath}
\usepackage{amsthm}
\theoremstyle{plain}
\newtheorem{thm}{Theorem}[section]

\newtheorem{lem}[thm]{Lemma}

\theoremstyle{definition}

\newtheorem{ex}[thm]{Example}

\newcommand{\End}{\operatorname{End}}

\newcommand{\Hom}{\operatorname{Hom}}
\newcommand{\Ann}{\operatorname{Ann}}
\newcommand{\MaxSpec}{\operatorname{MaxSpec}}
\newcommand{\Ker}{\textrm{Ker}}

\newcommand{\dive}{\operatorname{div}}

\newcommand{\0}{\bar 0}
\newcommand{\1}{\bar 1}

\numberwithin{equation}{subsection}

\def\Z{{\mathbb Z}}

\def\bbc{\mathbb C}
\def\:{\colon}

\def\h{\mathfrak{h}}

\newcommand{\fg}{\mathfrak{g}}

\def \fm{\mathfrak{m}}
\def \fg{\mathfrak{g}}

\def \fh{\mathfrak{h}}

\def \fh{\mathfrak{h}}

\def \sl{\mathfrak{sl}}

\def \fn{\mathfrak{n}}

\def\C{{\mathbb C}}

\def\Z{{\mathbb Z}}

\def\bu{\textbf{U}}

\newcommand{\lie}[1]{\mathfrak{#1}}

\def\bu{\textbf{U}}

\begin{document}
\normalsize

\title[On representations of Cartan   map Lie  superalgebras ]
{On representations of  Cartan    map Lie superalgebras}

\author{Irfan Bagci}
\address{Department of Mathematics \\
University of North Georgia \\
Oakwood, GA 30566}
\email{irfan.bagci@ung.edu}

%\author{Samuel Chamberlin}
%\address{Computer Science and Mathematics Department\\
%Park University\\
%Parkville, MO 64152.}
%\email{samuel.chamberlin@park.edu}

\begin{abstract}
Let $\fg$ be a simple Cartan type Lie superalgebra  and $A$ be a finitely generated,  commutative, associative algebra with unity over $\bbc$. We refer to the Lie superalgebras of the form $\fg\otimes A$ as Cartan map Lie superalgebras. This paper  describes finite dimensional irreducible representations of
  the Cartan map Lie superalgebras $\fg \otimes A$ and  twisted  Cartan map Lie superalgebras $(\fg \otimes A)^{T},$ where $T$ denotes  a finite abelian group acting on $\fg$ and the maximal ideal spectrum of  $A$ by automorphisms.
\end{abstract}
\maketitle
%\tableofcontents
\section{Introduction}

Recently there has been some interest in representation theory of the  map Lie (super)algebras $\fg \otimes A$ and twisted map Lie (super)algebras $(\fg \otimes A)^{T}$ where $\fg$ is a finite dimensional simple Lie (super)algebra and $A$ is a commutative associative algebra with unity over the field of complex numbers, and $T$ is a finite group acting both on $\fg$ and $A$ by automorphisms \cite{Lau, Li, ER2, RK, NSS, Sav}. Of particular interest is the classification of finite dimensional irreducible representations of these algebras.

In 1977, V. Kac provided a complete classification of simple Lie superalgebras, \cite{Kac}. The simple finite-dimensional Lie superalgebras are divided into two types based on their even part: they are either classical (simple classical Lie superalgebra are also of  two type: basic and strange) or of Cartan type.  Lie superalgebras of Cartan type consists of four infinite families $W(n),  S(n),  \tilde{S}(n)$,  and  $H(n)$ \cite{ Kac, Sch}.

%Let $\fg$ be a Lie superalgebra and $A$ be a commutative, associate, unital algebra. Then the superalgebras of the form $\fg \otimes A$ are called map superalgebras.
When  $A= \C[X_1^{\pm 1}, \dots, X_n^{\pm 1}]$, the Laurent polynomial algebra in $n$ commuting variables, the superalgebra $\fg \otimes A$ is called multiloop superalgebra.  A description of the finite-dimensional simple modules of the multiloop superalgebras of the basic  classical Lie superalgebras was given in \cite{ ER2, RK} and  the general case  in \cite{Sav}. The twisted case is also given  in \cite{Sav}.

%In this work we finish the classification of finite dimensional irreducible representations for map superalgebras of simple Lie superalgebras by classifying finite dimensional irreducible representations of  map superalgebras of Cartan type Lie superalgebras.
In this paper $\fg$ is a simple Cartan type Lie superalgebra and $A$ is a finitely generated, commutative, associative unitary algebra. The Lie superalgebra $\fg \otimes A$ is refered to as  Cartan map Lie superalgebra.
 %In \cite{BC} the authors define an integral form for $\fg \otimes A$ and give an explicit integral basis for the integral form.
  In this article we  classify irreducible finite dimensional representations of  the Lie superalgebras $\fg \otimes A$ and $(\fg \otimes A)^{T}$. We first introduce Generalized Kac Modules. These are a natural generalization of the Kac modules defined for $\fg$, \cite{Kac}. These are modules induced from the map Lie algebra $\fg_0 \otimes A$, where $\fg_0$ is the degree zero part of $\fg$ in the $\Z$-grading. We then show that all irreducible finite dimensional representations of $\fg \otimes A$  can be described as irreducible quotients of Generalized Kac Modules. For $\fg=  S(n),    \tilde{S}(n),$ and $ H(n)$ all irreducible finite dimensional representations are tensor products of evaluation representations. For $\fg= W(n)$ irreducible finite dimensional representations are parametrized by irreducible finite dimensional representations of the Lie algebra $\fg_0\otimes A$ with finite support. Lastly we give a complete description of irreducible finite dimensional representations of the twisted Cartan map Lie superalgebra $(\fg \otimes A)^{T}$ under the assumptions that $A$ is finitely generated and the group $T$ is abelian and acts freely on maximal ideal spectrum of $A$. We show that every irreducible finite dimensional representation of $(\fg \otimes A)^{T}$ is the restriction of an irreducible finite dimensional  representation of $\fg\otimes A$.

This paper is organized as follows: In Section 2 we fix some notation and briefly review basic facts about simple Lie superalgebras of Cartan type, introduce Cartan map Lie superalgebras and record the properties we are going to need in the rest of the paper. Next in Section 3 we define Generalized Kac Modules and prove some basic properties of these modules.  In Section 4 we state and prove  the main results of this paper.

%\textbf{Note:}
\section{Preliminaries}

\subsection{}Throughout this paper $\C$ is the set  of complex numbers. All vector spaces, algebras and tensor products are defined  over  $\C$. Whenever we refer to dimension of an algebra or ideal, we mean its dimension as vector space over $\C$. We will denote the integers by $\Z$, the positive integers by $\Z^+$, and we let $\Z_2=\{\0, \1\}$ denote the quotient ring $\Z/2\Z$..

 A Lie superalgebra is a finite dimensional $\Z_2$-graded vector space $\fg=\fg_{\0}\oplus\fg_{{\1}}$ with a bracket $[,]:\fg\otimes\fg\rightarrow\fg$ which preserves the $\Z_2$-grading and satisfies graded versions of the operations used to define Lie algebras. The even part $\fg_{\0}$ is a Lie algebra under the bracket.

Given any Lie superalgebra $\fg$ let $\bu(\fg)$ be the universal enveloping superalgebra of $\fg$.  $\bu(\fg)$ admits a PBW type basis and if   $x_1, \cdots, x_m$ is  a basis of $\fg_{\0}$ and $y_1, \dots, y_n$ is a basis of $\fg_{\1}$, then the elements
$$x_1^{i_1} \dots x_m^{i_m}y_1^{j_1}\dots y_n^{j_n}\ \  \text{with} \ \ i_1, \dots, i_m \geq 0  \ \ \text{and} \ \ j_1, \dots, j_n \in \{0, 1\}$$
form a basis of the universal enveloping superalgebra $\bu(\fg)$.

%A simple Lie superalgebra $\fg=\fg_0\oplus\fg_1$ is called \emph{classical} if its even part $\fg_0$ is a reductive Lie algebra.
Finite-dimensional complex simple Lie superalgebras were classied by Kac \cite{Kac}.  The finite-dimensional simple Lie superalgebras are divided into two classes: the classical (when the even part is a reductive Lie algebra)
and the Cartan ones (otherwise). Lie superalgebras of Cartan type consists of four infinite families $$W(n), \ n\geq 2; \ S(n), \ n\geq 3; \  \tilde{S}(n), \ n\geq 4  \  \text{ and even}; \  H(n), \ n \geq 4.$$

 We briefly define Cartan type Lie superalgebras. Assume that $n\geq 2$.  Let $\Lambda(n)$ denote the exterior algebra on $n$ odd generators $\xi_1, \dots, \xi_n$; $\Lambda(n)$ is a $2^n$-dimensional associative algebra we assign to it a $\Z$-grading by setting deg  $\xi_i = 1$ for $1\leq i\leq n$. The $\Z_2$-grading is inherited from the $\Z$-grading by setting $\Lambda(n)_{\0} = \bigoplus_{k} \Lambda^{2k}(n)$ and $\Lambda(n)_{\1} = \bigoplus_{k} \Lambda^{2k+1}(n)$.

Let $\bar x$ denote the parity of a homogeneous  element $x$ in a $\Z_2$-graded vector space. A (homogeneous) \emph{superderivation} of $\Lambda(n)$ is a linear map $D: \Lambda(n) \to \Lambda(n)$ which satisfies $D(xy)=D(x)y + (-1)^{\bar D \; \bar x}xD(y)$ for all homogenous $x,y \in \Lambda(n)$. % where $\deg$ denotes the degree of a homogeneous element in %$\Z_2$-grading.
 Then $W(n)$ is the Lie superalgebra of super derivations of $\Lambda(n)$. The bracket on  $W(n)$ is the supercommutator bracket. That is,
$$[x,y]=xy-(-1)^{ \bar x \bar y}yx$$
for homogeneous $x,y$ and extended to all $W(n)$ by linearity.  The $\Z$-grading on $\Lambda(n)$ induces a $\Z$-grading on $W(n)$
$$W(n)= W(n)_{-1} \oplus W(n)_0 \oplus \dots \oplus W(n)_{n-1},$$
where $W(n)_k$ consists of derivations that increase the degree of a homogeneous element by $k$. The $\Z_2$-grading on $W(n)$ is inherited from the $\Z$-grading by setting $W(n)_{\0}= \oplus_{k} W(n)_{2k}$ and $W(n)_{\1}= \oplus_{k} W(n)_{2k+1}.$
 Since the bracket preserves  the $\Z$-grading $W(n)_0$ is a Lie subalgebra of $W(n)$ and  $W(n)_0 \cong \mathfrak{gl}(n).$

Denote by  $\partial_i$ the derivation of $\Lambda(n)$ defined by
$$\partial_i(\xi_j): = \delta_{i,j}.$$ Then any element $D$ of $W(n)$ can be written in the form
$$\sum_{i=1}^nf_i\partial_i,$$
where $f_i \in \Lambda(n)$.

The superalgebra $S(n)$ is the subalgebra of $W(n)$ consisting of all elements $D \in W(n)$ such that
$\dive(D) = 0$, where
$$\dive\left(\sum_{i=1}^n f_i \partial_i\right): = \sum_{i=1}^n \partial_i(f_i).$$
The superalgebra $S(n)$ has a $\Z$-grading induced by the grading of $W(n)$
$$S(n)= S(n)_{-1} \oplus S(n)_0 \oplus \dots \oplus S(n)_{n-2}.$$
and $S(n)_0$ is isomorphic to $\sl(n)$.

The simple Lie superalgebra $\tilde{S}(n)$ is defined only when $n$ is even. The superalgebra $\tilde{S}(n)$ has a $\Z$-grading
$$\tilde{S}(n)= \tilde{S}(n)_{-1} \oplus \tilde{S}(n)_0 \oplus \dots \oplus \tilde{S}(n)_{n-2}.$$
For each $r$ with $0\leq r \leq n-2, \tilde{S}_r(n)= S_r(n)$ . The difference is that $\tilde{S}(n)_{-1}$ has a basis consisting of $\xi_1 \dots \xi_n \partial_i$.
This grading is not an algebra grading.

The subspace of $W(n)$ spanned by all super derivations of the form
$$\Sigma _{i\in \Z^{+}}\partial _i(f)\partial _i$$
where $f\in \Lambda(n)$, is a Lie superalgebra called $\tilde{H}(n)$. It inherits a natural $\Z$-grading from $W(n)$ and we have
$$\tilde{H}(n)=\tilde{H}(n)_{-1} \oplus \tilde{H}(n)_0 \oplus \dots \oplus \tilde{H}(n)_{n-2} .$$
The subalgebra $$H(n)=[\tilde{H}(n), \tilde{H}(n)]=H(n)_{-1} \oplus H(n)_0 \oplus \dots \oplus H(n)_{n-3}$$ is a simple Lie superalgebra of Cartan type.
$H(n)_0 \cong \mathfrak{so}(n)$  as a Lie algebra %and the homogeneous component $H(n)_r$ %is isomorphic as an $H(n)_0$-module to $\Lambda(n)_{r+2}$ via
%$D_f\mapsto f$. %Thus the superderivations $x_I=D_{\xi _I}$, where $\empty \neq I\subseteq %N$, form a basis for $\tilde{H}(n)$, and
%$\dim H(n)_r=\binom{n}{r+2}$.

Throughout this work unless otherwise is noted  $\fg$ will be a Cartan type Lie superalgebra. The crucial difference between the Cartan type superalgebras and the classical superalgebras
is that the $\fg_{\0}$ component is no longer reductive.  However, as was described above, the Cartan type Lie superalgebras  admit a ${\mathbb Z}$-grading: $\fg=\bigoplus_{k\in {\mathbb Z}} \fg_{k}.$  The grading is compatible with the $\Z_{2}$-grading in the sense that $\bigoplus_{k}\fg_{2k}=\fg_{\0}$ and $\bigoplus_{k}\fg_{2k+1}=\fg_{\1}$.
%Since  the bracket respects the grading (i.e., \ $[\fg_{i},\fg_{j}] \subseteq \fg_{i+j}$ for all %integers $i,j$)  $\fg_0$ is a Lie subalgebra of $\fg$ .  $\fg_0 \cong \mathfrak{gl}(n)$ if $\fg= %W(n)$,  $\fg_0 \cong \mathfrak{sl}(n)$ if $\fg= S(n)$ or $\tilde{S}(n)$, and $\fg_0 \cong %\mathfrak{so}(n)$ if $\fg= H(n)$

Let $$\fg^+ := \fg_1 \oplus \dots \oplus \fg_s$$
 so that $\fg$ has a decomposition
$$\fg = \fg_{-1} \oplus \fg_0 \oplus \fg^+.$$
Since the bracket on $\fg$ preserves the $\Z$-grading, $\fg_{0} \oplus \fg^+ $ and $ \fg_{-1}\oplus \fg_0$ are subalgebras of $\fg$.

 The  Cartan subalgebra $\fh$ of $\fg$ coincides with Cartan subalgebra of $\fg_0$.  Let  $\fg_{0} = \fn_{0}^- \oplus \h \oplus \fn_{0}$ of $\fg_{0}$ be a triangular decomposition of $\fg_0$. This gives a triangular decomposition
$$\fg = \fn^- \oplus \h \oplus \fn,$$
where $\fn := \fn_{0} \oplus \fg^+$ and $\fn^- := \fn_{0}^{-} \oplus \fg_{-1}$. This  decomposition is also called the distinguished triangular decomposition.
Let $\Phi =\Phi_{\0}\cup \Phi_{\1}$  be the roots of $\fg$ with respect to this triangular decomposition and $\Delta =\Delta_{\0}\cup \Delta_{\1}$  be the set of simple roots.  Let $\Phi^+ =\Phi_{\0}^+\cup \Phi_{\1}^+$ be the set of positive roots and  $\Phi^- = - \Phi^+$ be the set of  negative roots.  %$\Phi_{\0}^{+,-} =\Phi_{\0}\cap \Phi^{+,-}$

\subsection{}
Let $\lie{g}$ be a Cartan type Lie superalgebra  with super commutator bracket  $[\ ,\ ]_{\lie{g}}$. Fix a commutative associative unitary algebra $A$.  The \emph{map superalgebra} of $\lie{g}$ is the $\Z_2$-graded vector space  $\lie{g}\otimes A$, where   $(\lie{g}\otimes A)_{\0} = \lie{g}_{\0}\otimes A $ and  $(\lie{g}\otimes A)_{\1}= \lie{g}_{\1}\otimes A $, with bracket given by extending the bracket
$$[z\otimes a, z'\otimes b]=[z,z']_{\lie{g}}\otimes ab,\ z,z'\in\fg,\ a,b\in A.$$
by linearity. We will refer to a Lie superalgebra of this form as  \emph{Cartan map Lie superalgebra}.%$\lie{g}$ can be embedded in this Lie superalgebra as $\lie{g}\otimes 1$.

%If $A$ is the coordinate ring of the scheme $X$ then the Lie superalgebra $\lie{g}\otimes A$ is the Lie superalgebra of regular functions on $X$ with values in $\lie{g}$ with point-wise bracket, \cite{Sav}.
\begin{ex} If we  take $A=\C[X]$, then  the Lie superalgebra  $\fg \otimes \C[X]$ is called a \emph{current superalgebra}.

If we  take $A=\C[X, X^{-1}]$, then  the Lie superalgebra $\fg \otimes \C[X, X^{-1}]$ is called a \emph{loop superalgebra}.

If we  take $A=\C[X_1^{\pm 1}, \dots,  X_n^{\pm 1}]$, then  the Lie superalgebra $\fg \otimes \C[X_1^{\pm 1}, \dots,  X_n^{\pm 1}]$ is called a \emph{multiloop superalgebra}.

\end{ex}

\subsection{} Recall that $\fg =\fn^- \oplus \fh \oplus \fn$.
Let $$\fg \otimes A=\fn^- \otimes A \oplus \fh \otimes A \oplus \fn \otimes A$$ be the corresponding triangular decomposition for $\fg \otimes A$.

A representation $(\pi, M)$ of $(\fg \otimes A)$  is called a \emph{highest weight representation} if there exists a nonzero vector $m \in M$ such that:
$$U(\fg \otimes A)m=M,$$
$$(\fn \otimes A)m=0,$$
$$U(\fh \otimes A)m= \C m.$$

Such a vector $m$  is called a \emph{highest weight vector}.

%The following lemma is easy to verify.
\begin{lem}\label{lemhw}Every finite dimensional irreducible representation $(\pi, M)$   of  $\fg \otimes A$ is a highest weight representation.
\end{lem}
\begin{proof} Since the Cartan subalgebra $\fh$ of $\fg$ is a Cartan subalgebra of the Lie algebra $\fg_0$ it is an abelian Lie algebra. Since both $\fh$ and $A$ are commutative, it follows that $\fh \otimes A$ is also an abelian Lie algebra. Since all irreducible finite dimensional representations of $\fh \otimes A$ are one dimensional, $M_{\beta} \neq 0$ for some $\beta \in \fh^{\ast}$. Note that $(\fg_{\alpha}\otimes A)M_{\beta} \subseteq M_{\alpha+\beta}$ for all $\alpha \in \Phi.$ Now the simplicity of $M$ implies that $M$ is a weight module. Since $M$ is finite dimensional there exists a maximal weight $\lambda \in \fh^{\ast}$ of $M$ and let $m$ be a nonzero vector in $M_{\lambda}$. Now it follows  that $m$ generates $M$ and $(\fn \otimes A)m=0$. Therefore $(\pi, M)$ is a highest weight representation.

\end{proof}

An ideal $J$ of $A$ of will be called an ideal of \emph{finite codimension} if $A/J$ is a finite dimensional vector space.
\begin{lem}\label{lemideal} Let $(\pi, M)$ be a finite dimensional irreducible representation of $(\fg\otimes A)$. Then there exists an ideal $J$ of $A$ of finite codimension such that
$(\fg\otimes J)M=0$.
\end{lem}

\begin{proof} By  Lemma \ref{lemhw} $M$ is a highest weight module. Let $\mu$ be the highest weight of $M$, $M_{\mu}$ be the weight space of weight $\mu$ and $\beta \in \Phi^+$. Consider the map
$L: A \longrightarrow \Hom_{\C}(M_{\mu}\otimes \fg_{-\beta}, M_{\mu-\beta})$ defined by
$$L(a)(m\otimes x)= (x\otimes a)m, a \in A, m \in M_{\mu}, x \in \fg_{-\beta}.$$
Note that $L$ is a linear map. Let $J_{\beta}: = \Ker (L)$. One can easily verify that $J_{\beta}$ is an ideal of $A$. Since $M$ and $\fg$ are finite dimensional, it follows that  that $J_{\beta}$ is an ideal of $A$ of  finite codimension for any $\beta \in \Phi^+$.
Let $$J:= \bigcap_{\beta \in \Phi^+}J_{\beta}.$$
Note that this intersection is a finite intersection as $\fg$ is finite dimensional it has only finitely many positive roots. $J$ and $J_{\beta}$ are ideals of $A$ and
$J\subseteq J_{\beta}$. Then by the third isomorphism theorem for rings it follows that $J_{\beta}/J$ is an ideal of $A/J$ and $A/J \cong (A/J_{\beta})/(J_{\beta}/J)$. Since $A/J_{\beta}$ is finite dimensional, it follows from the isomorphism that $A/J$ is finite dimensional. Thus $J$ is an ideal  of  finite codimension.

We then have $(\fn^{-} \otimes J)M_{\mu}=0$. Since  $M$ is a highest weight module of high weight $\mu$,  $(\fn \otimes J)M_{\mu}=0$.
Since $\fh \otimes J \subseteq [\fn^{-} \otimes J, \fn \otimes J]$, it follows that $(\fh \otimes J)M_{\mu}=0.$ Thus we have
$(\fg \otimes J)M_{\mu}=((\fn^- \oplus \fh \oplus \fn) \otimes J)M_{\mu}=0.$

Let $N= \{m \in M|(\fg \otimes J)m=0\}$. Since $M_{\mu} \neq 0$, $N$ is a nonempty subset of $M$.  Using the fact that $(\fg \otimes J)$ is an ideal of $(\fg \otimes A)$, we see that
$N$ is a submodule of $M$. Since $M$ is irreducible, it follows that $N=M$. That is $(\fg \otimes J)M=0$

\end{proof}

\subsection{}

Recall that we assume  $A$ is finitely generated. Let $\MaxSpec (A)$ denote the maximal ideal spectrum of $A$ and $\fg$ be a finite dimensional Lie (super)algebra.

Let $\fm_1, \dots, \fm_r \in \MaxSpec (A)$ be pairwise distinct and $n_1, \dots, n_r \in \Z^{+}$. Let $\varphi_{\fm_1^{n_1}, \dots,  \fm_r^{n_r}}$ be the algebra map defined by composing the map
$$ \fg \otimes A \longrightarrow (\fg\otimes A)/(\fg\otimes (\fm_1^{n_1} \dots \fm_r^{n_r}))$$
and the isomorphism

$$(\fg\otimes A)/(\fg\otimes (\fm_1^{n_1} \dots \fm_r^{n_r})) \cong  (\fg \otimes A/\fm_{1}^{n_1}) \oplus \dots \oplus (\fg \otimes A/\fm_{r}^{n_r}).$$

That is, $$\varphi_{\fm_1^{n_1}, \dots,  \fm_r^{n_r}}: \fg \otimes A \longrightarrow \bigoplus_{i=1}^r (\fg \otimes A/\fm_{i}^{n_i}).$$

For $i=1, \dots, r$, let $(\pi_i, M_i)$ be a finite dimensional irreducible representation of $(\fg \otimes A/\fm_{i}^{n_i})$. Then the composition of $\varphi_{\fm_1^{n_1}, \dots, \fm_r^{n_r}}$ with the map
%-module withe the corresponding representation $\pi_i : \fg \otimes A/\fm_{i}^{n_i} \longrightarrow \End(S_i)$.

$$ \bigoplus_{i=1}^r (\fg \otimes A/\fm_{i}^{n_i})\longrightarrow \End (M_1\otimes \dots \otimes M_r )$$
is a representation of $\fg \otimes A$ and is denoted by

\begin{equation}\label{evrep}
\varphi_{\fm_1^{n_1}, \dots, \fm_r^{n_r}}(\pi_1, \dots, \pi_r).
\end{equation}
We denote the corresponding module by

\begin{equation}\label{evmod}
\varphi_{\fm_1^{n_1}, \dots, \fm_r^{n_r}}(M_1, \dots, M_r).
\end{equation}
Representations of the form $\varphi_{\fm_i}(M_i)$  are called \emph{evaluation representations} and corresponding modules are called \emph{evaluation representations}.

%Irreducible finite dimensional representations for map algebras of semisimple and reductive Lie algebras were classified.The following result was established by ....

\begin{lem}\label{lemred} \cite{CFK, Li, NSS} \label{fdred} Let $\fg$ be a reductive Lie algebra. We let $\fg'=[\fg, \fg]$  be its semisimple part and $Z(\fg)$ denote its center, so that $\fg= \fg' \oplus Z(\fg)$.  Then all irreducible finite dimensional representations of $\fg\otimes A$ are of the form $\rho \otimes \pi$, where $\rho \in (Z(\fg)\otimes A)^{\ast}$ is an irreducible finite dimensional representation of abelian Lie algebra $Z(\fg)\otimes A$ and $\pi$ is an irreducible finite dimensional representation of $\fg'\otimes A$. Such $\pi$ are precisely the irreducible evaluation representations of $\fg'\otimes A$.
\end{lem}

\begin{lem}\cite[Proposition 8.4]{Che} \label{Che} Let $C$ and $D$ be Lie superalgebras. Every irreducible $C\otimes D$ module $V$ is either isomorphic to $V_C \otimes V_D$, where $V_C$ (resp. $V_D)$ is an irreducible $C$ (resp.$C$) module, or it is isomorphic to a proper subspace of $V_C \otimes V_D$ .

\end{lem}

Let $(\pi, M)$ be a representation of $(\fg \otimes A)$.  Let $\Ann_{A}(M)$ denote the annihilator ideal  of $M$ in $A$, i.e.
$\Ann_{A}(M)=\{a \in A| (\fg \otimes a)M=0\}$. By support of $M$ or $\pi$ we mean the support of the ideal $\Ann_{A}(M)$.

\begin{lem} \label{fsupp} Let $\fg$ be a reductive Lie algebra or a simple  Cartan type Lie superalgebra. Let $(\pi, M)$ be a finite dimensional irreducible representation of $(\fg\otimes A)$.
\begin{itemize}
\item[(a)]$(\pi, M)$ is a  representation  of the form (\ref{evrep}) if and only if the ideal  $\Ann_{A}(M)$ has finite support.
\item[(b)]$(\pi, M)$ is a tensor product of evaluation  representations if and only if $\Ann_{A}(M)$ is a radical ideal and has finite support.
\end{itemize}
\end{lem}

\begin{proof} We prove part (a). Part (b) is argued similarly.

If $\pi$ is of the from (\ref{evrep}), then $\pi= \varphi_{\fm_1^{n_1}, \dots, \fm_s^{n_s}}(\pi_1, \dots, \pi_s)$ for some pairwise distinct $\fm_1, \dots, \fm_s \in MaxSpec (A)$, ${n_1},\dots, n_s \in \Z^{+}$, and some representation $\pi_1, \dots, \pi_s$.
Then it is clear that the  ideal $J= \fm_1^{n_1} \cdots \fm_s^{n_s}$ has finite support and $\pi(\fg \otimes J)=0$. Thus $\pi$ has finite support.

Suppose that $\pi(\fg \otimes J)=0$ for some ideal $J$ of finite support.  Then the radical ideal of $J$ is  $\fm_1 \cdots \fm_s$ for some distinct  $\fm_1, \dots, \fm_s \in \MaxSpec (A)$. Since $A$ is finitely generated every ideal of $A$ contains some power of its radical \cite[Proposition 7.17]{AM}. Thus the ideal $(\fm_1 \cdots \fm_s)^{l}=\fm_1^{l} \cdots \fm_s^{l} \subseteq J$ for some $l \in \Z^{+}$. Then $\pi$ factors through the map $$\varphi_{\fm_1^{l},  \cdots,  \fm_s^{l}}: \fg \otimes A \to \bigoplus_{i=1}^s (\fg \otimes A/\fm_{i}^{l}).$$

It follows from Lemma \ref{Che} that finite dimensional irreducible representations  of $\bigoplus_{i=1}^s (\fg \otimes A/\fm_{i}^{l})$ are of the form $(\pi_1, M_1) \otimes \cdots \otimes (\pi_s, M_s)$, where $(\pi_i, M_i)$ is a finite dimensional irreducible representation of $\fg \otimes A/\fm_{i}^{l}$. Therefore  $\pi $ is of the form (\ref{evrep}).

\end{proof}

\section{Generalized Kac Modules}

In this section we define Generalized Kac modules and prove some basic properties of these modules.

\subsection{}Recall that  we assume  $A$ is finitely generated and  we fixed a  decomposition  $\fg =\fg_{-1} \oplus \fg_0 \oplus \fg^+$.

Let $V$ be a finite dimensional irreducible $(\fg_0 \otimes A)$-module.  Consider  the induced representation of $\fg\otimes A$,

$$K(V):= U(\fg\otimes A) \otimes_{U((\fg_0 \oplus \fg^+)\otimes A)}V$$

where $V$ is viewed as $((\fg_0 \oplus \fg^+)\otimes A)$ -module via inflation through the canonical quotient map$((\fg_0 \oplus \fg^+)\otimes A)\to (\fg_0 \otimes A)$. We call $K(V)$ \emph{Generalized Kac Module}. The reason we call these Generalized Kac Module is that these are precisely  Kac modules when $A= \C$(cf. \cite{Kac}).

%The following lemma is easy to verify.

As a $\Z_2$-graded vector space $K(V) \cong \Lambda(\fg_{-1}\otimes A)\otimes V$. The $\Z$-grading on $\Lambda(\fg_{-1}\otimes A)$ induces a $Z$-grading on $K(V)$
\begin{equation*}
K(V)= K(V)_0 \oplus K(V)_1\oplus \dots =\bigoplus_{k \geq 0} \Lambda^k(\fg_{-1}\otimes A)\otimes V.
\end{equation*}
%One can easily check that this grading is compatible with the grading in $\fg \otimes A$ and with the grading in $U(\fg \otimes A)$ which is induced from $\fg \otimes A$.

Since $\fg \otimes A$ is a Lie superalgebra by the PBW Theorem we can identify $\fg \otimes A$ with its image in $U= U(\fg \otimes A)$.  The standard  filtration $\{U_i\}$ of $U$ is defined as follows
$U_0 =\C, U_1= \C + (\fg \otimes A),$  and  $U_i = U_1^i;$  that is, $U_i$ is the span of all products $u_1u_2\dots u_i$ with $u_l \in U_1.$

Let $U_+= \bigoplus_{i>0} U_i$. Then $U_+$ is a subalgebra of $U$ and as a vector space
\begin{equation}\label{eqkv} K(V)=(1\otimes V) \oplus (U_+(\fg_{-1}\otimes A)\otimes V)
\end{equation}
\begin{lem} \label{gkm} Let $V$ be a finite dimensional irreducible $(\fg_0 \otimes A)$-module and  $K(V)$ be the associated Generalized Kac Module.

\begin{itemize}
\item[(a)]$K(V)$ contains a unique proper maximal $\Z$-graded submodule $N(V).$
\item[(b)] $L(V):= K(V)/N(V)$ is an irreducible $(\fg\otimes A)$-module.
\item[(c)] $L(V_1) \cong L(V_2)$ as $(\fg\otimes A)$-modules if and only if $V_1 \cong V_2$ as $(\fg_0\otimes A)$-modules.
\item[(d)]$\Ann_A(L(V))= \Ann_A(V)$.

\end{itemize}
\end{lem}

\begin{proof}

Let $K(V)^+:= \bigoplus_{k>0}K(V)_k$. Note that  $K(V)^+$ is a proper subspace of $K(V)$. Let $N(V)$  be a maximal submodule of $K(V)$. Part (a) follows from the following claim.

\vspace{.1 in}

 \textbf{Claim:}  $N(V)$ is a proper graded subspace of $K(V)^+$.

\vspace{.1 in}

  We assume the contrary that $N(V) \not\subset K(V)^+$. Then we can choose an element $w \in N(V)$, where $w= w_0+ \dots +w_n$ with $w_k \in K(V)_k$ such that $w_0 \neq 0$.

In view of the irreducibility of $V$ and the definition of an induced module, for every $k>0$ there must exist a $u_k \in U_k((\fg_{-1}\oplus \fg_0)\otimes A)$ such that such that $w_k =u_k(w_0)$, i.e.
$$w= w_0 +(\sum_{k=1}^nu_k)(w_0)=(1+u)w_0.$$ From the properties of grading it follows that $u$ is a nilpotent operator on $K(V)$. Therefore there exists a $u' \in  U((\fg_{-1}\oplus \fg_0)\otimes A)$ which is an inverse for
$1+u$ on $K(V)$, and we have $w_0=u'w \in N(V)$. But then $N(V)\supset K(V)$, which is a contradiction. This proves the claim that $N(V) \subset K(V)^+$.

 We now prove that $N(V) =\bigoplus_{k\geq 0} (N(V)\cap K(V)_k)$. We construct a submodule $N(V)'\supset N(V)$ as follows: $N(V)' = \bigoplus N(V)_k$, where $N(V)_k$ is the projection of $N(V)$ on the subspace $K(V)_k$. By part (a), $N(V)_0^{'}=0$. Thus $N(V)^{'} \neq K(V)$. Since $N(V)$ is maximal, it follows that $N(V)= N(V)^{'}$, and $N(V)\cap K(V)_k = N(V)_k$.

Part (b) follows from part (a). Parts (c) and (d) follow from definitions.

\end{proof}

\begin{lem} \label{lemvlv}Let $V$ be a finite dimensional irreducible $(\fg_0\otimes A)$-module and  $L(V)$ be the unique irreducible quotient of $K(V)$. Let  $J$ be an ideal of $A$. Then $(\fg_0\otimes J)V=0$ if and only if $(\fg\otimes J)L(V)=0$.
\end{lem}

\begin{proof} Note first that $(\fg\otimes J)L(V)=0$ implies that  $(\fg_0\otimes J)V=0$.

Let $J$ be an ideal of $A$ such that $(\fg_0\otimes J)V=0$.  Since $\fg \otimes J$ is an ideal of $\fg \otimes A$, it follows that
$(\fg\otimes J)K(V)$ is a submodule of $K(V)$.

\vspace{.1in}

\textbf{Claim :} $(\fg\otimes J)K(V)$ is a proper submodule of $K(V)$.

\vspace{.1in}

Since the $\Z$-grading on $\fg$ is an algebra grading when $\fg= W(n), S(n), H(n)$, it follows that   $[\fg_{-1}, \fg_{-1}]=0$. One can also directly verify that    $[\fg_{-1}, \fg_{-1}]=0$ when $\fg= \tilde{S}(n)$. Now using the fact that $[\fg_{-1}, \fg_{-1}]=0$ and (\ref{eqkv}) one has
\begin{equation}\label{e1}
(\fg_0\otimes J)K(V)= (\fg_0\otimes J)(1\otimes V)+ (\fg_0\otimes J)(U_+(\fg_{-1}\otimes A)\otimes V) \subseteq (\fg_{-1}\otimes J)K(V)
\end{equation}

Using the fact that $\fg_{-1}$ is abelian  and (\ref{eqkv}) one has
\begin{equation}\label{e2}
(\fg^+\otimes J)K(V)= (\fg^+\otimes J)U(\fg_{-1}\otimes A) \otimes V \subseteq U(\fg_{-1}\otimes A)(\fg_{-1}\otimes J)K(V) \subseteq (\fg_{-1}\otimes J)K(V).
\end{equation}

Putting together (\ref{e1}) and (\ref{e2}) one has

\begin{equation*}
(\fg \otimes J)K(V)= (\fg_{-1}\otimes J)K(V) + (\fg_0\otimes J)K(V)+ (\fg^+\otimes J)K(V)=(\fg_{-1}\otimes J)K(V).
\end{equation*}

Since $K(V)\cong U(\fg_{-1}\otimes A)\otimes V$ as a $\Z_2$-graded vector space, it follows that $(\fg_{-1}\otimes J)K(V)$ is a proper submodule of $K(V)$. This completes the proof of the claim and completes the proof of the other implication.
\end{proof}

\begin{lem} \label{lemevrep} Let $V$ be a finite dimensional irreducible $(\fg_0\otimes A)$-module and $L(V)$ be the unique irreducible quotient of $K(V)$.
\begin{itemize}
\item[(a)]$L(V)$ is a tensor product of modules of the form (\ref{evrep}) if and only if $V$ is.
\item[(b)]$L(V)$ is a tensor product of evaluation  modules if and only if $V$ is.
\end{itemize}
\end{lem}
\begin{proof}We prove part (a). One can argue part (b) similarly.

By Lemma \ref{fsupp}  $L(V)$ is of the form  (\ref{evrep}) if and only if  $L(V)$ has finite support. Since $\Ann_{A}(L(V))= \Ann_{A}(V)$ by Lemma \ref{gkm} (d),  the result follows.
\end{proof}

\begin{lem}\label{fd} Let $V$ be a finite dimensional irreducible $(\fg_0\otimes A)$-module and $L(V)$ be the unique irreducible quotient of $K(V)$. Then $L(V)$ is finite dimensional if and only if $V$ is of the form (\ref{evrep}).
\end{lem}

\begin{proof}If $L(V)$ is finite dimensional, then by Lemma \ref{lemideal} $(\fg \otimes J)L(V)=0$ for some ideal $J$ of finite codimension.  Since $A/J$ is a finite dimensional $\C$-algebra, $A/J$ is an Artin ring. An Artin ring has only a finite number of maximal ideals \cite[Proposition 8.3]{AM}.  Thus $J$ has finite support. Now  by Lemma \ref{fsupp} $L(V)$ is of the form (\ref{evrep}) and by Lemma \ref{lemevrep} so is $V$.

Suppose that $V$ is of the form (\ref{evrep}). Then  by Lemma \ref{lemideal}  there exists an ideal $J$ of finite support such that $(\fg \otimes J)V=0$. Since $A/J$ has finitely many maximal ideals and $A$ is finitely generated, $A/J$ is finite dimensional. Thus $J$ is an ideal of finite codimension.
Now by Lemma \ref{lemideal} $(\fg \otimes J)L(V)=0$.  It follows from the definition of the Generalized Kac Module that $L(V)$ is a quotient of  the induced module $U(\fg\otimes A/J) \otimes_{U((\fg_0 \oplus \fg^+)\otimes A/J)}V$. Note that this induced module is isomorphic to $U(\fg_{-1}\otimes (A/J))$ as $\Z_2$-graded vector spaces. Since $\fg_{-1}$ and $A/J$ are finite dimensional vector spaces $\fg_{-1}\otimes (A/J)$ is a finite dimensional vector space and as a $\Z_2$-graded vector space its odd, i.e., it has no even part. Now the result follows from the PBW Theorem for Lie superalgebras.
\end{proof}

%\begin{lem} Let $\fm_1, \dots, \fm_r \in \MaxSpec A$ be pairwise distinct,  $n,1, \dots, n_r \in \Z^{+}$ and $S_i$ be a finite dimensional $(\fg \otimes A/\fm_{i}^{n_i})$-module. Then

%$$K(\varphi_{\fm_1^{n_1}, \dots, \fm_r^{n_r}}(S_1, \dots, S_r))\cong K(\varphi_{\fm_1^{n_1}}(S_1))\otimes \dots \otimes K(\varphi_{\fm_1^{n_r}}(S_r)).$$

%\end{lem}

\section{Classification of Finite dimensional irreducible representations}

Recall that we assume $A$ is a finitely generated algebra. In this section we state and prove main results of this paper.

\subsection{}

 Recall that $\fg$ admits a $\Z$-grading $\fg = \fg_{-1}\oplus \fg_{0} \oplus \fg^+$ and we fixed the distinguished triangular decomposition  $\fg = \fn^- \oplus \h \oplus \fn$. Recall that   $\fg_0 \cong \mathfrak{gl}(n) $  for $W(n)$, $\fg_0 \cong \mathfrak{sl}(n) $ for $S(n)$ and $\tilde{S}(n)$, and $\fg_0 \cong \mathfrak{so}(n) $ for $H(n)$.
Note that $\fg_0$ is reductive Lie algebra when $\fg=W(n)$ and is a simple Lie algebra for other three types. Let $\fg_0{'}:=[\fg_0, \fg_0]$ denote the semisimple part of $\fg_0$ and $Z(\fg_0)$ be its center.

Our first result is the classification of finite dimensional irreducible representations of $\fg \otimes A$:

\begin{thm} Let $\fg$ be a simple Cartan type Lie superalgebra. Then all irreducible finite dimensional representations of $\fg \otimes A$ are of the form (\ref{evrep}).

\begin{itemize} \label{tuntwisted}

\item[(a)] Let $\fg=W(n)$. Then irreducible finite dimensional representations of $\fg \otimes A$ are parametrized  by $(\rho, \pi)$ where $\rho$ is a finite dimensional irreducible representation of $Z(\fg_0) \otimes A$ with finite support and $\pi$ is a finite dimensional irreducible representation of $\fg_0'\otimes A$. %Then all irreducible finite dimensional representations of $\fg \otimes A$ are of the form (\ref{evrep}).
\item[(b)]Let $\fg=S(n), \tilde{S}(n)$ or $H(n)$. Then irreducible finite dimensional representations of $\fg \otimes A$ are parametrized by the finite dimensional irreducible representations of $\fg_0\otimes A$.
\end{itemize}
\end{thm}

\begin{proof}

 Let $\psi \in (\fh \otimes A)^{\ast}$  and $\C_{\psi}$ be a one dimensional representation of $\fh \otimes A$ via $\psi$. Let $\fn \otimes A$ act trivially on $\C_{\psi}$. Consider the induced module (Verma module)
$$ U(\fg \otimes A)\otimes _{U((\fh\oplus \fn)\otimes A)}\C_{\psi}.$$
By standard  arguments one can show that this induced module is a highest weight module and has a unique maximal submodule. Thus has a unique irreducible quotient $L(\psi)$.
This implies that every irreducible highest weight $(\fg \otimes A)$-module is isomorphic to $L(\psi)$ for some $\psi \in (\fh \otimes A)^{\ast}$.

Let $L$ be a finite dimensional irreducible $\fg \otimes A$-module. Then by Lemma \ref{lemhw} $L$ is a highest weight module. Thus $L \cong L(\psi)$ for some $\psi \in (\fh \otimes A)^{\ast}$. Let $v$ be a highest weight of $L(\psi)$ and $\fg_0^{'}$ denote the semisimple part of $\fg_0$( Note that $\fg_0^{'}=\fg_0 $ when  $\fg = S(n), \tilde{S}(n)$ or $H(n)$). Consider the induced module
$$V:= U(\fg_0^{'} \otimes A)v.$$
Note that $V$  is a finite dimensional $(\fg_0^{'} \otimes A)$-module. Now we prove that $V$ is irreducible as $(\fg_0^{'} \otimes A)$-module:
Recall that $\fg = \fn^- \oplus \h \oplus \fn$. Let $\fg_0^{'} = (\fn_{0}^{'})^- \oplus \h \oplus \fn_0^{'}$ be the triangular decomposition induced by the one on $\fg$. Then
$$V= U((\fn_{0}^{'})^-\otimes A)v.$$
Let $w \in V$ be a weight vector. Since $L$ is irreducible as a $\fg \otimes A$-module, there exists some $Y \in  U((\fn_{0}^{'})^-\otimes A)U(\fg^+\otimes A)$ such that $Yw=v$.
Since $U(\fg^+\otimes A)v=0$, this is possible only when $ Y \in  U((\fn_{0}^{'})^-\otimes A)$.  Thus $V$ is irreducible as $(\fg_0^{'} \otimes A)$-module.

By Lemma \ref{lemideal} $(\fg \otimes J)L=0$ for some ideal $J$ of finite codimension. By Lemma \ref{lemvlv}, it follows that  $(\fg_0^{'}\otimes J)V=0$.
For $\fg=W(n)$, $\fg_0$ has one dimensional center  $Z(\fg_0)$ . The abelian Lie algebra $Z(\fg_0)\otimes A$ acts on $V$ as scalars. Furthermore, $\fg^+\otimes A$ acts trivially on $V$ as $[\fg^+\otimes A,  \fg_0^{'} \otimes A]\subseteq \fg^+\otimes A$.

By the universal property of induced modules it follows that  $L$ is the  quotient of Generalized Kac Module $K(V)$. That is,  $L \cong L(V)$.

By Lemma \ref{fd}, it follows that  $L(V)$ is of the form (\ref{evrep}).

 \begin{itemize}

 \item[(a)] By  Lemma \ref{lemevrep} $V$ is of the from (\ref{evrep}). Since $\fg_0$ is a reductive Lie algebra  Lemma  \ref{lemred} implies that $V$ is of the form $\rho \otimes \pi$ where $\rho$ is a one dimensional representation of the abelian Lie algebra $Z(\fg_0) \otimes A$ an $\pi$ is an irreducible representation of  $\fg_0' \otimes A$. Note that $\Ann_A(L(V))=\Ann_A(V) = \Ann_A(\rho) \cap \Ann_A(\pi)$ and this implies that
     the support of $V$ is the union of supports of $\rho$ and $\pi$. Since $\pi$ is a tensor product of evaluation representations the support of $\pi$ is finite. Thus support of $V$ will be finite only when support of $\rho$ is finite. Now the result follows from  Lemma \ref{fsupp}.
\item[(b)]   Since $\fg_0$  is a simple Lie algebra Lemma  \ref{lemred}  implies that $V$ is a tensor product of evaluation representations. Now the result follows from Lemma \ref{lemevrep}(d).
\end{itemize}
\end{proof}

\subsection{}

Let $T$ be a finite group acting on $\fg$ and  $\MaxSpec (A)$ by automorphisms. We assume that Lie superalgebra automorphisms respect the $\Z_2$-grading.  Then $T$ acts on the map superalgebra  $\fg \otimes A$ by linearly extending the action
$t(x\otimes a)= t x \otimes t a, t \in T, x \in \fg, a \in A.$ The associated \emph{twisted map  superalgebra} is the Lie superalgebra of the fixed points $(\fg \otimes A)^{T} \subseteq \fg \otimes A$. That is,

$$(\fg \otimes A)^{T} = \{ z \in \fg\otimes A| t z=z, \  \forall t \in T\}.$$
We will refer to a Lie superalgebra of this form as  \emph{twisted Cartan map Lie superalgebra}.

Recall that $A$ is finitely generated. In this section we assume that $T$ is finite abelian group and the action of $T $ on $\MaxSpec (A)$ is free. The isomorphism classes of irreducible representations of the abelian group $T$ can be identified with the character group of $T $.  Let $X(T)$ be the character group of $T$. $X(T)$ is an abelian group.

Since $T$ acts on algebras $\fg$ and $A$ by automorphisms, each algebra has a unique decomposition into isotypic components. Therefore $$\fg = \bigoplus_{\xi \in X(T)}\fg_{\xi}, \  A = \bigoplus_{\xi \in X(T)}A_{\xi}.$$

Since $(\fg_{\xi}\otimes A_{\xi^{'}})^{T}= \{0\}$ unless $\xi^{'}= -\xi$ we get

$$ (\fg \otimes A)^{T} =\bigoplus_{\xi \in X(T)}(\fg_{\xi}\otimes A_{-\xi})^{T}.$$

%The following result was established by Savage in \cite{Sav} for map superalgebras of classical Lie superalgebras. Same proof works for map superalgebras of Cartan type Lie superalgebras. Since the proof is the same we will skip

%\begin{lem} Let $\fg$ be a finite dimensional simple Lie superalgebra. Then all ideals of $(\fg \otimes A)^{T}$ are of the form $(\fg \otimes J)^{T}$, where $J$ is a $T$-invariant ideal of $A$.
%\end{lem}

\begin{lem} \label{lres} Every finite dimensional irreducible representation of $(\fg \otimes A)^{T}$ is the restriction of a finite dimensional irreducible representation of $(\fg \otimes A)$.
  %Let  $(\pi, M)$ be a  finite dimensional irreducible representation of $(\fg \otimes A)^{T}$. Then  $(\pi, M)$ is the restriction of a representation $(\tilde{\pi}, \tilde{M})$ of  $(\fg \otimes A)$. Moreover, $(\pi, M)$ is irreducible if and only if $(\tilde{\pi}, \tilde{M})$ is.
\end{lem}

\begin{proof} The proof is essentially the same as for map superalgebras of basic classical Lie superalgebras (see, for example proof of Proposition 8.5 in \cite{Sav}), and will be skipped here.

\end{proof}

Let $\varphi_{\fm_1^{n_1}, \dots, \fm_r^{n_r}}(\pi_1, \dots, \pi_r)$ be as in (\ref{evrep}). We denote the restriction of $\varphi_{\fm_1^{n_1}, \dots, \fm_r^{n_r}}(\pi_1, \dots, \pi_r)$ to $(\fg \otimes A)^{T}$ by

\begin{equation}\label{resevrep}
\varphi_{\fm_1^{n_1}, \dots, \fm_r^{n_r}}^{T}(\pi_1, \dots, \pi_r).
\end{equation}

Our second result is the classification of finite dimensional irreducible representations of $(\fg \otimes A)^{T}$:

\begin{thm} Let $\fg$ be a simple Cartan type Lie superalgebra  and $A$ be a finitely generated,  commutative, associative algebra with unity. Let  $T$ be a finite abelian group acting on $\fg$ and $\MaxSpec (A)$ by automorphisms. Suppose that    the action of $T $ on $\MaxSpec (A)$ is free. Then all irreducible finite dimensional representations of $(\fg \otimes A)^{T}$ are of the form (\ref{resevrep}) such that the maximal ideals $\fm_1, \dots, \fm_r$ are distinct and all lie in distinct $T$-orbits,  and  the representations $\pi_1, \dots, \pi_r$ are irreducible.
\end{thm}

\begin{proof} Note that Lemma \ref{lres} and Theorem \ref{tuntwisted} implies that  every irreducible finite dimensional representation of $(\fg \otimes A)^{T}$ is of the from (\ref{resevrep}). If $\pi_i$ is not irreducible for some $i$, then it is clear that (\ref{resevrep}) will not be irreducible. Thus the representations $\pi_i$ are irreducible. One can always assume that maximal ideals are distinct by replacing $\pi_i$ by appropriate tensor products.

$T$ acts on isomorphism classes of finite dimensional irreducible representation of $\fg$ by $t [\pi]:= [\pi \circ t^{-1}],$
where $t \in T$ and $[\pi]$ is the isomorphism class of the representation $\pi$ of $\fg$.  For any $t \in T$ one can easily see that $\pi_i \circ t^{-1}$ is a representation of $\fg \otimes (A/t \fm_i^{n_i})$.
Note that
$$\varphi_{t \fm_i^{n_i}}(\pi_i \circ t^{-1})= \pi_i \circ t^{-1} \circ \varphi_{t \fm_i^{n_i}} \cong \pi_i \circ t^{-1} \circ (t \circ \varphi_{ \fm_i^{n_i}}\circ t^{-1} ) \cong \varphi_{ \fm_i^{n_i}}(\pi_i) \circ t^{-1}.   $$

That is, we showed that $\varphi_{t \fm_i^{n_i}}(\pi_i \circ t^{-1}) \cong \varphi_{ \fm_i^{n_i}}(\pi_i) \circ t^{-1}.$ Since $t^{-1}$ fixes $(\fg \otimes A)^{T}$ we have $\varphi_{t \fm_i^{n_i}}^{T}(\pi_i \circ t^{-1})\cong \varphi_{ \fm_i^{n_i}}^{T}(\pi_i).$  Thus we can assume that the maximal ideals $\fm_i$ all lie in distinct $T$-orbits.

\end{proof}

\end{document}